\documentclass[12pt]{article}
\usepackage{extarrows}
\usepackage{multirow}%
\usepackage{amsmath,amssymb,amsfonts}%
\usepackage{enumitem}
\usepackage{amsthm}%
\usepackage{mathrsfs}%
\allowdisplaybreaks[4]

\usepackage[numbers,sort&compress,comma,square]{natbib}

\usepackage{hyperref}

\hypersetup{
	colorlinks=true,
	anchorcolor=yellow,
	linkcolor=cyan,
	filecolor=blue,
	urlcolor=cyan,
	citecolor=green,
}
\usepackage{cleveref}
\crefformat{equation}{Eq.~(#2#1#3)}

\usepackage[margin=1in]{geometry}

\theoremstyle{definition}
\newtheorem{defn}{Definition}[section]
\newtheorem{quest}{Question}

\theoremstyle{plain}
\newtheorem{thm}[defn]{Theorem}
\newtheorem{lem}[defn]{Lemma}
\newtheorem{cor}[defn]{Corollary}

\newtheorem{prop}[defn]{Proposition}
\theoremstyle{remark}

\newtheoremstyle{case}{}{}{}{}{}{:}{ }{}
\theoremstyle{case}

\DeclareMathOperator{\ii}{\mathrm{i}}
\DeclareMathOperator{\N}{\mathbb{N}}
\DeclareMathOperator{\R}{\mathbb{R}}
\DeclareMathOperator{\Z}{\mathbb{Z}}
\DeclareMathOperator{\Q}{\mathbb{Q}}
\DeclareMathOperator{\D}{\mathbf{D}}
\DeclareMathOperator{\IOCG}{\mathrm{IOCG}}
\DeclareMathOperator{\PST}{\mathrm{PST}}
\DeclareMathOperator{\MST}{\mathrm{MST}}
\DeclareMathOperator{\UST}{\mathrm{UST}}

\DeclareMathOperator{\Sp}{Sp}

\begin{document}
\title{\textbf{Quantum state transfer on integral oriented circulant graphs}}

\author{Xing-Kun Song
\footnote{Email address: \href{mailto:xksong@126.com}{xksong@126.com} (X.-K. Song)} \\[2mm]
  \small School of Mathematics, East China University of Science and Technology, \\
  \small Shanghai 200237, P.R. China\\}
\date{}
\maketitle

\begin{abstract}
  An oriented circulant graph is called integral if all eigenvalues of its Hermitian adjacency matrix are integers. The main purpose of this paper is to investigate the existence of perfect state transfer ($\PST$ for short) and multiple state transfer ($\MST$ for short) on integral oriented circulant graphs. Specifically, a characterization of $\PST$ (or $\MST$) on integral oriented circulant graphs is provided. As an application, we also obtain a closed-form expression for the number of integral oriented circulant graphs with fixed order having $\PST$ (or $\MST$).
\end{abstract}

\begin{flushleft}
\end{flushleft}

\section{Introduction}

All graphs considered in this paper have neither loops nor multiple edges. Let $\Gamma=(V, E, A)$ be a \emph{mixed graph} (introduced by Harary and Palmer~\cite{HP66}) with vertex set $V$, undirected edge set $E$, and arc (directed edge) set $A$. In particular, we say that $\Gamma$ is \emph{oriented} (resp. \emph{undirected}) if it contains only directed (resp.~undirected) edges. The \emph{Hermitian adjacency matrix} of $\Gamma$, introduced by Liu and Li~\cite{LL15}, and independently by Guo and Mohar~\cite{GM17}, is defined as $H_{\Gamma}=(h_{uv})_{u,v\in V}$, where
\begin{equation*}
  h_{uv} =
  \begin{cases}
    1,    & \text{ if }	(u,v)\in E, \\
    \ii,  & \text{ if } (u,v)\in A, \\
    -\ii, & \text{ if } (v,u)\in A, \\
    0,    & \text{ otherwise}.
  \end{cases}
\end{equation*}
Here $\ii=\sqrt{-1}$. Since $H_{\Gamma}$ is a Hermitian matrix, all its eigenvalues called the \emph{Hermitian eigenvalues} of $\Gamma$ are real. The multiset of Hermitian eigenvalues of $\Gamma$ is called the \emph{Hermitian spectrum} of $\Gamma$, and denoted by $\Sp_H(\Gamma)$. For more results on Hermitian eigenvalues of mixed graphs, we refer the reader to \cite{Mo16,YWGQ20}.

A mixed graph is called \emph{integral} if all its Hermitian eigenvalues are integers. The problem of characterizing integral (undirected) graphs was proposed by Harary and Schwenk~\cite{HS74} in 1974. It has been discovered  that integral graphs can play a role in the so-called perfect state transfer (defined below) in quantum spin networks. See~\cite{BCRSS02} for a survey on integral graphs.

The concept of perfect state transfer was introduced by Bose~\cite{Bo03} in 2003. Let $\Gamma$ be a mixed graph, and let $H_\Gamma$ be the Hermitian adjacency matrix of $\Gamma$. We say that $\Gamma$ has \emph{perfect state transfer} ($\PST$ for short) from $u$ to $v$ if there exists a time $t \in \R$ and a complex unimodular scalar $\gamma$ such that
\begin{equation}\label{eq:::1}
  U(t)\mathbf{e}_u= \gamma \mathbf{e}_v,
\end{equation}
where $U(t) = \exp(\ii t H_\Gamma)$ is the \emph{transition matrix} of $H_{\Gamma}$, and $\ii = \sqrt{-1}$. Here $\Gamma$ is called \emph{phase} of $\PST$. In particular,  if $u = v$ in \Cref{eq:::1}, we say that $\Gamma$ is \emph{periodic at vertex $u$}. Furthermore, if $U(t)$ is a scalar multiple of the identity matrix, then $\Gamma$ is \emph{periodic}.
In recent years, the study of $\PST$ on graphs has aroused a great deal of interest, and it is well known that graphs having $\PST$ are rare. In 2011--2012, Godsil~\cite{Go08,Go12a,Go12b} provided some basic properties for periodicity and perfect state transfer of graphs. For more details on $\PST$, we refer the reader to~\cite{CG21b}.

In quantum informatics and quantum computing, there has been tremendous interest in $\PST$ on Cayley graphs. Let $G$ be a finite group with identity element $e$, and let  $S$ be a subset of $G\setminus \{e\}$. The \emph{Cayley graph} $\mathrm{Cay}(G,S)$ is defined as the graph with vertex set $G$ and arc set $E=\{(x,y)\in G\times G : yx^{-1} \in S\}$. In particular, if $S=S^{-1}$ then  $\mathrm{Cay}(G,S)$ is an undirected graph, and if $S\cap S^{-1}=\emptyset$ then $\mathrm{Cay}(G,S)$ is an oriented graph. In ~\cite{CCL20,CFT21,CF21,TFC19}, Cao et al. investigated $\PST$ on Cayley graphs over abelian groups or dihedral groups. For a comprehensive survey about $\PST$ on Cayley graphs, we refer the reader to \cite[Chapter 9]{LZ18}.

A circulant graph is a Cayley graph over a cyclic group. Let $\mathbb{Z}_n$ be the additive group of integers module $n$, and let $\mathcal{C}$ be a subset of $\mathbb{Z}_n\setminus \{0\}$. The \emph{circulant graph} $G(\mathbb{Z}_n,\mathcal{C})$ has vertex set $\mathbb{Z}_n$ and arc set $\{(a,b): b-a  \in \mathcal{C}, a,b\in \mathbb{Z}_n\}$. Here  $\mathcal{C}$ is called the \emph{symbol} of $G(\mathbb{Z}_n,\mathcal{C})$. In 2003, So~\cite{So06} characterized all integral undirected circulant graphs. Based on this result, Ba\v{s}i\'{c}~\cite{Ba13} gave a characterization for $\PST$ on integral undirected circulant graphs. For more results about $\PST$ on integral undirected circulant graphs, see~\cite{ANOPRT10,BP09,BPS09,PB11,SSS07}. With regard to oriented circulant graphs, it is natural to ask the following question.
\begin{quest}\label{que::1}
  Which oriented circulant graphs have $\PST$?
\end{quest}
According to Godsil and Lato~\cite{GL20}, if an oriented graph has $\PST$, then all its Hermitian eigenvalues are integers or integer multiples of $\sqrt{\Delta}$, where $\Delta$ is a square-free integer. In this paper, our first goal is to give an answer to \Cref{que::1} for integral oriented circulant graphs.

Very recently, Kadyan and Bhattacharjya~\cite{MB21a} provided a characterization for integral oriented circulant graphs.
\begin{thm}{\bf (See~\cite[Theorem 5.3]{MB21a})}\label{thm::IOCG}
  Let $\Gamma=G(\mathbb{Z}_n,\mathcal{C})$ be an oriented circulant graph.
  \begin{enumerate}[label = \bf(\roman*)]
    \item If $n\not\equiv 0 \pmod 4$, then $\Gamma$ is
          integral if and only if $\mathcal{C}=\emptyset$.
    \item If $n\equiv 0 \pmod 4$, then $\Gamma$ is
          integral if and only if $\mathcal{C}= \bigcup\limits_{d\in \D}S_n(d)$, where $\D \subseteq \{ d: d\mid n/4\}$, and $S_n(d) = G_n^{1}(d)$ or $S_n(d)=G_n^{3}(d)$.
  \end{enumerate}
  Here $G_n^r(d)=\{ dk: k\equiv r \pmod 4, \gcd(dk,n )= d\}=d G_{n/d}^r(1)$.
\end{thm}

Let $G(\mathbb{Z}_n,\mathcal{C})$ be a non-empty integral oriented circulant graph.  By~\Cref{thm::IOCG}, we see that $n\equiv 0 \pmod 4$, and the symbol $\mathcal{C}= \bigcup_{d\in \D}S_n(d)$ corresponds to  the mapping $\sigma:\D\rightarrow  \{1,-1\}$ where $\sigma(d)=1$ if $S_n(d) = G_n^{1}(d)$ and $\sigma(d)=-1$ if $S_n(d)=G_n^{3}(d)$. Therefore, each non-empty integral oriented circulant graph $G(\mathbb{Z}_n,\mathcal{C})$ is determined by its order $n$ ($n\equiv 0 \pmod 4$),  a set $\D$ of positive divisors of $\frac{n}{4}$, and a mapping from $\D$ to $\{1,-1\}$. For this reason, we use $\IOCG_n(\D,\sigma)$ instead of $G(\mathbb{Z}_n,\mathcal{C})$ in what follows.  For example, if $G(\mathbb{Z}_8,\mathcal{C})=\IOCG_8(\{1,2\},\{1,-1\})$, then $G_8^{1}(1)=\{k:k\equiv 1\pmod 4, \gcd(k,8)=1\}=\{1,5\}$, $G_8^{3}(2)=2G_4^{3}(1)=2\{k:k\equiv 3\pmod 4, \gcd(k,4)=1\}=2\{3\}=\{6\}$, and hence $\mathcal{C}=G_8^{1}(1)\cup G_8^{3}(2)=\{1,5,6\}$.

Let $\vartheta_2(n)$ be the largest positive integer $\alpha$ such that $2^{\alpha} \mid n$. We define $\D_i = \{d\in \D \mid \vartheta_2(n/d)=i, 0\leq i\leq \vartheta_2(n)\}$, where $\D\subseteq \D_n=\{d : d \mid n, 1 \leq d < n\}$. The first result of this paper is as below.

\begin{thm}\label{thm::1.6}
  Let $\Gamma=\IOCG_n(\D,\sigma)$ be an integral oriented circulant graph. Then the following two statements are equivalent:
  \begin{enumerate}[label = \bf(\roman*)]
    \item $\Gamma$ has $\PST$ between vertices $b+n/2$ and $b$, for all $b \in \Z_n$;\label{thm1.2.1}
    \item $n\in 4\N$ and $\D_2=\{n/4\}$.\label{thm1.2.2}
  \end{enumerate}
\end{thm}

Up to now, there are few results about $\PST$ on oriented graphs. For undirected graphs, Kay~\cite{Ka11} proved that $\PST$ (whose adjacency matrix is real symmetric) occurs only between two non-disjoint pairs of vertices. Unlike undirected graphs, some oriented graphs having $\PST$ between multiple vertices were found. Here, we introduce some results in recent years. We say that a graph admits \emph{universal state transfer} ($\UST$ for short) if there exists $\PST$ between every pair of vertices. Cameron et al.~\cite{CFGHST14} proposed this definition and showed that only the complete graphs $K_2$ and $K_3$ with complex Hermitian adjacency matrices which have $\UST$. Connelly et al.~\cite{CGKST17} conjectured that $K_3$ is the only nontrivial unweighted oriented graph with $\UST$. It is clear that graphs having $\UST$ are more rare than $\PST$. For this reason, Godsil and Lato~\cite{GL20} proposed the concept of multiple state transfer. A graph is said to have \textit{multiple state transfer} ($\MST$ for short) if it contains a vertex subset $S$ such that $\PST$ occurs between each pair of vertices in $S$.  Also, they gave some examples of graphs having $\MST$.
For some results of $\PST$ and $\MST$ on oriented graphs, we refer the reader to \cite{La19}.
In~\cite{La19}, Lato asked the following question.

\begin{quest}\label{que::2}
  Can we build infinite families of graphs with $\MST$?
\end{quest}

In the second part of this paper, we focus on studying~\Cref{que::2}, and obtain a characterization for integral oriented circulant graphs having $\MST$.

\begin{thm}\label{thm::1.7}
  Let $\Gamma=\IOCG_n(\D,\sigma)$ be an integral oriented circulant graph. Then the following two statements are equivalent:
  \begin{enumerate}[label = \bf(\roman*)]
    \item $\Gamma$ has $\MST$ between vertices $b$, $b+n/4$, $b+n/2$, $b+3n/4$, for all $b \in \Z_n$;\label{thm1.3.1}
    \item $n\in 8\N$, $\D_2= \{n/4\}$, and $\D_3= \{n/8\}$.\label{thm1.3.2}
  \end{enumerate}
\end{thm}

The paper is organized as follows. In~\Cref{sec::2}, we provide an exact formula of the Ramanujan’s sine sum, and give an expression for  the eigenvalues of integral oriented circulant graphs. In~\Cref{sec::3}, we give the proof of \Cref{thm::1.6}. In~\Cref{sec::4}, we give the proof of \Cref{thm::1.7}.

\section{The eigenvalues of integral oriented circulant graphs}\label{sec::2}

In the section, we will give a characterization of the eigenvalues of integral oriented circulant graphs. For further research, we need to introduce the concepts of Ramanujan's sum and Ramanujan’s sine sum.

Let $n$ be a positive integer, and let $d$ be a divisor of $n$. Recall that $G_{n}(d)=\{k:1\leq k\leq n-1, \gcd(k, n) =d\}$ and $G_n^r(d)=\{ dk: k\equiv r \pmod 4, \gcd(dk,n )= d\}=d G_{n/d}^r(1)$.

\begin{defn}{\bf (See~\cite[p.~308]{HW08})}
  Let $q$ and $n$ be positive integers. Define
  \begin{equation}\label{eq::a}
    c_{n}(q)=\sum_{\substack{1 \leq a \leq n \\ \gcd(a, n)=1}} e^{2 \pi \ii \frac{a}{n}q}=\sum_{a\in G_n(1)} \cos \frac{2 \pi a q}{n}.
  \end{equation}
  The expression for $c_{n}(q)$ is known as \emph{Ramanujan's sum}.
\end{defn}

\begin{prop}\label{prop::2.2}
  Let $c_n(q)$ be the Ramanujan's sum. Then
  \begin{enumerate}[label = \bf(\roman*)]
    \item $c_1(q)=1$, for all positive integers $q$;
    \item if $n$ is a prime number,
          \[
            c_n(q) =
            \begin{cases}
              -1,  & \text{if $n\nmid q$}, \\
              n-1, & \text{if $n\mid q$}.
            \end{cases}
          \]
  \end{enumerate}
\end{prop}

A set $\mathcal{C} \subseteq G_n(1)$ is called \textit{skew-symmetric} if $n-a \in \mathcal{C}^{-1} $ for all $a\in \mathcal{C}$, where $\mathcal{C}^{-1}= G_n(1)\setminus \mathcal{C}$. Now, the Ramanujan's sum \Cref{eq::a} can also be written as
\begin{equation}\label{ramasum2}
  \begin{split}
    c_n(q)= \sum_{a\in G_n(1)} w_n^{aq}= \sum_{a\in \mathcal{C}} 2 \cos \frac{2\pi a q}{n},
  \end{split}
\end{equation}
where $\omega_n=\exp(\ii 2\pi/n)$ is the $n$-th root of unity.

Note that the Ramanujan's sum $c_n(q)$ is an integer, for any $q,n \in \N$. We are replacing cosine with sine in \Cref{ramasum2}, obtain

\begin{displaymath}
  s_{n}^\mathcal{C}= \sum_{a\in \mathcal{C}} \frac{\omega^{aq}_n-\omega^{-aq}_n}{\ii} =\sum_{a\in \mathcal{C}} 2 \sin \frac{2 \pi a q}{n},
\end{displaymath}
where $\omega_n=\exp(\ii 2\pi/n)$ is the $n$-th root of unity.

In~\cite{MB21a}, Kadyan and Bhattacharjya proved that, $s_{n}^\mathcal{C}$ is still an integer, for any $q$, $n\equiv 0 \pmod 4 \in \N$ if and only if $\mathcal{C}=G_n^1(1)$ or $G_n^3(1)$. To match sign $\sigma$, we adjust the sign in \cite{MB21a} to redefine the following.
\begin{defn}{\bf (See~\cite{MB21a})}
  Let $q$ and $n\equiv 0 \pmod 4$ be positive integers. Define

  \begin{displaymath}
    s_{n}^\sigma(q)=\ii \sum_{a\in S_n(1)} (\omega^{aq}_n-\omega^{-aq}_n) =-\sum_{a\in S_n(1)} 2 \sin \frac{2 \pi a q}{n},
  \end{displaymath}
  where $\omega_n=\exp(\ii 2\pi/n)$, if $S_n(1) = G_n^{1}(1)$ then $\sigma=1$ or if $S_n(1)=G_n^{3}(1)$ then $\sigma=-1$. The expression for $s_{n}^\sigma(q)$ is called as \emph{Ramanujan's sine sum}.
\end{defn}

Note that $s_{n}^{1}(q)=-s_{n}^{-1}(q)$, for any $q$, $n\equiv 0 \pmod 4 \in \N$.
Without loss of generality, we only consider  $s_{n}^{1}(q)$, for short denoted by $s_{n}(q)$. Now, we have $s_{n}^\sigma(q)=\sigma s_{n}(q)$.

In the following, we will give a characterization of Ramanujan's sine sum $s_{n}(q)$. First, we will prove some lemma.

\begin{lem}\label{lem::2.1}
  Let $n=4m$ and with $m$ being an odd positive integer. Then
  \[
    m+4r\in
    \begin{cases}
      G_{4m}^{1}(1), & \text{if $m\equiv 1\pmod4$}, \\
      G_{4m}^{3}(1), & \text{if $m\equiv 3\pmod4$},
    \end{cases}
  \]
  for some $r\in G_{m}(1)$.
\end{lem}

\begin{proof}
  Assume that $m\equiv 1\pmod4$. Let $r\in G_{m}(1)$. Then $\gcd(r,m)=1$, which implies that $\gcd(m+4r,4m)=1$, and $m+4r\equiv 1\pmod4$. Thus, $m+4r\in G_{4m}^{1}(1)$. Similarly, if $m\equiv 3\pmod4$, then $m+4r\in G_{4m}^{3}(1)$.
\end{proof}

\begin{lem}\label{clm::1}
  Let $n=4m$ and with $m$ being an odd positive integer. Then
  \begin{equation}\label{eq:1}
    s_{n}(q)=
    \begin{cases}
      (-1)^{(m-1)/2}(-1)^{(q+1)/2}2c_{m}(q), & \text{if $q$ is odd,} \\
      0,                                     & \text{otherwise}.
    \end{cases}
  \end{equation}
\end{lem}

\begin{proof}
  By \Cref{lem::2.1}, it follows that
  \[
    \begin{aligned}
      s_{n}(q) & = \ii\sum_{a\in G_{n}^{1}(1)}(\omega_{n}^{aq}-\omega_{n}^{-aq})                                                                                                   \\
               & = \begin{cases}
                     \ii\sum_{a\in G_{4m}^{1}(1)}(\omega_{4m}^{aq}-\omega_{4m}^{-aq}),  & m\equiv 1\pmod4 \\
                     -\ii\sum_{a\in G_{4m}^{3}(1)}(\omega_{4m}^{aq}-\omega_{4m}^{-aq}), & m\equiv 3\pmod4
                   \end{cases}                                              \\
               & = \begin{cases}
                     \ii\sum_{m+4r\in G_{4m}^{1}(1)}(\omega_{4m}^{(m+4r)q}-\omega_{4m}^{-(m+4r)q}),  & m\equiv 1\pmod4 \\
                     -\ii\sum_{m+4r\in G_{4m}^{3}(1)}(\omega_{4m}^{(m+4r)q}-\omega_{4m}^{-(m+4r)q}), & m\equiv 3\pmod4
                   \end{cases}                                          \\
               & = \begin{cases}
                     \ii\sum_{r\in G_{m}(1)}(\omega_{4}^{q}\omega_{m}^{rq}-\omega_{4}^{-q}\omega_{m}^{-rq}),  & m\equiv 1\pmod4 \\
                     -\ii\sum_{r\in G_{m}(1)}(\omega_{4}^{q}\omega_{m}^{rq}-\omega_{4}^{-q}\omega_{m}^{-rq}), & m\equiv 3\pmod4
                   \end{cases} \\
               & = (-1)^{(m-1)/2}\ii\sum_{r\in G_{m}(1)}(\ii^{q}\omega_{m}^{rq}-\ii^{-q}\omega_{m}^{-rq})                                                                          \\
               & = (-1)^{(m-1)/2}(\ii^{q+1}c_{m}(q)-\ii^{1-q}c_{m}(-q))                                                                                                            \\
               & = (-1)^{(m-1)/2}(-1)^{(q+1)/2}(1-(-1)^{-q})c_{m}(q)                                                                                                               \\
               & = \begin{cases}
                     (-1)^{(m-1)/2}(-1)^{(q+1)/2}2 c_{m}(q), & \text{if $q$ is odd,} \\
                     0,                                      & \text{otherwise}.
                   \end{cases}
    \end{aligned}
  \]
  This proves \Cref{clm::1}.
\end{proof}

\begin{lem}\label{clm::2}
  Let $n=2^{t}m$ and with $m$ being an odd positive integer and $t\geq 2$. Then
  \begin{equation}\label{eq:2}
    s_{n}(q)=2^{t-2}s_{4m}(q/2^{t-2}),
  \end{equation}
  for some $q/2^{t-2}\in \N$.
\end{lem}
\begin{proof}
  We prove the lemma by using induction on $t$. It is clear that the identity holds for the case $t = 2$.
  Let $t\geq 3$ and assume that the identity holds for each $t=k$. Now let $n'=2^{k+1}m$.
  Since
  \[
    G_{2^{t+1}m}^{1}(1)=G_{2^{t}m}^{1}(1)\cup (2^{t}m+G_{2^{t}m}^{1}(1)),
  \]
  we have
  \[
    \begin{aligned}
      s_{n'}(q) & =\ii \sum_{a\in G_{2^{k+1}m}^{1}(1)}(\omega_{2^{k+1}m}^{aq}-\omega_{2^{k+1}m}^{-aq})                                                                                                             & \\
                & =\ii \sum_{a\in G_{2^{k}m}^{1}(1)\cup (2^{k}m+G_{2^{k}m}^{1}(1))}(\omega_{2^{k+1}m}^{aq}-\omega_{2^{k+1}m}^{-aq})                                                                                & \\
                & =\ii \sum_{a\in G_{2^{k}m}^{1}(1)}(\omega_{2^{k+1}m}^{aq}-\omega_{2^{k+1}m}^{-aq})+\ii \sum_{a\in 2^{k}m+G_{2^{k}m}^{1}(1)}(\omega_{2^{k+1}m}^{aq}-\omega_{2^{k+1}m}^{-aq})                      & \\
                & =\ii \sum_{a\in G_{2^{k}m}^{1}(1)}(\omega_{2^{k+1}m}^{aq}-\omega_{2^{k+1}m}^{-aq})+\ii\sum_{a\in G_{2^{k}m}^{1}(1)}(\omega_{2^{k+1}m}^{(a+2^{k}m)q}-\omega_{2^{k+1}m}^{-(a+2^{k}m)q})            & \\
                & =\ii \sum_{a\in G_{2^{k}m}^{1}(1)}(\omega_{2^{k+1}m}^{aq}-\omega_{2^{k+1}m}^{-aq})+\ii\sum_{a\in G_{2^{k}m}^{1}(1)}(\omega_{2}^{q}\omega_{2^{k+1}m}^{aq}-\omega_{2}^{-q}\omega_{2^{k+1}m}^{-aq}) & \\
                & =\ii \sum_{a\in G_{2^{k}m}^{1}(1)}((1+(-1)^{q})\omega_{2^{k+1}m}^{aq}-(1+(-1)^{-q})\omega_{2^{k+1}m}^{-aq})                                                                                      & \\
                & \xlongequal{2\mid q} 2\ii\sum_{a\in G_{2^{k}m}^{1}(1)}(\omega_{2^{k}m}^{aq/2}-\omega_{2^{k}m}^{-aq/2})                                                                                           & \\
                & =2s_{2^{k}m}(q/2)                                                                                                                                                                                & \\
                & =2^{k-1}s_{4m}(q/2^{k-1}),
    \end{aligned}
  \]
  for some $q/2^{k-1}\in \N$. By the induction hypothesis the identity holds. This proves \Cref{clm::2}.
\end{proof}

Let $q'=q/2^{t-2}\in \N$. By substituting \Cref{eq:1} into \Cref{eq:2}, we can obtain \Cref{thm::1.4}.

\begin{thm}\label{thm::1.4}
  Let $n=2^{t}m$ and with $m$ being an odd positive integer and $t\geqslant 2$. Then
  \[
    s_{n}(q)=
    \begin{cases}
      (-1)^{(m-1)/2}(-1)^{(q'+1)/2}2^{t-1}c_{m}(q'), & \text{if $q'$ is odd,} \\
      0,                                             & \text{otherwise},
    \end{cases}
  \]
  where $q'=q/2^{t-2}\in \N$.
\end{thm}

In the remainder of this section, we will characterize the eigenvalues of integral oriented circulant graphs using Ramanujan's sine sum.

Let $\Gamma=G(\mathbb{Z}_n,\mathcal{C})$ be an integral oriented circulant graph, and let $H$ be the  Hermitian adjacency matrix of $\Gamma$. According to~\cite{MB21a}, the eigenvalues and eigenvectors of $H$ are, respectively, given by
\begin{equation}\label{eq::1}
  \mu_j=\ii \sum_{k \in \mathcal{C}} (\omega^{jk}_n-\omega^{-jk}_n),
  \quad \mathbf{v}_j=[1 \ \omega_n^k \ \omega_n^{2k} \cdots
      \omega_n^{(n-1)k}],
\end{equation}
for $0\leq j\leq n-1$, where $\omega_n=\exp(\ii 2\pi/n)$ is the $n$-th root of unity.

\vspace{5mm}

\begin{thm}\label{thm::1.5}
  Let $\Gamma=\IOCG_n(\D,\sigma)$ be an integral oriented circulant graph. Then the eigenvalues of $\Gamma$ are
  \begin{equation*}
    \mu_j=
    \begin{cases}
      \sum_{d\in \D_i}\sigma(d) (-1)^{(\frac{n}{2^{i}d}-1)/2} (-1)^{(\frac{j}{2^{i-2}}+1)/2} 2^{i-1} c_{\frac{n}{2^{i}d}} (j/2^{i-2}), & \text{if $j/2^{i-2}$ is odd,} \\
      0,                                                                                                                               & \text{otherwise,}
    \end{cases}
  \end{equation*}
  for $2\leq i \leq \vartheta_2(n)$, $0\leq j \leq n-1$, where $\D=\bigcup_{i=2}^{\vartheta_2(n)}\D_i \subseteq \left\{ d: d\mid n/4\right\}$.
\end{thm}

\begin{proof}
  Let $\Gamma=\IOCG_n(\D,\sigma)$ be an integral oriented circulant graph, and let $H$ be the Hermitian adjacency matrix of $\Gamma$. Then the eigenvalues $\mu_j$ of $H$ can be expressed in terms of Ramanujan's sine sum as follows

  \begin{equation} \label{eq::2}
    \begin{aligned}
      \mu_j & =\ii \sum_{k \in \bigcup_{d\in \D}S_n(d)} (\omega^{jk}_n-\omega^{-jk}_n)     \\
            & =\sum_{d\in \D}\ii \sum_{k\in S_n(d)} (\omega^{jk}_n-\omega^{-jk}_n)         \\
            & =\sum_{d\in \D}\ii \sum_{k\in dS_n(1)} (\omega^{jk}_n-\omega^{-jk}_n)        \\
            & =\sum_{d\in \D}\ii \sum_{k\in S_n(1)} (\omega^{jk}_{n/d}-\omega^{-jk}_{n/d}) \\
            & =\sum_{d\in \D} s^{\sigma}_{n/d}(j)                                          \\
            & =\sum_{d\in \D}\sigma(d) s_{n/d}(j),
    \end{aligned}
  \end{equation}
  for $0\leq j \leq n-1$, where $\D \subseteq \left\{ d: d\mid n/4\right\}$. For the Ramanujan’s sine sum $s_{n/d}(j)$, let $n/d=2^{i}m$ and $j=2^{i-2}j'$, with $m$ and $j'$ be an odd positive integer. If we fix the integer $i$, then $n/d$ and $j$ are determined by $i$. Let   $\D=\bigcup_{i=2}^{\vartheta_2(n)}\D_i \subseteq \left\{ d: d\mid n/4\right\}$ for $2\leq i \leq \vartheta_2(n)$.
  For \Cref{eq::2}, if j is fixed, then $i$ is determined by $j$, and it follows that $d\in \D_i$. Thus, we have
  \begin{equation} \label{eq::3}
    \mu_j=\sum_{d\in \D_i} \sigma(d) s_{n/d}(j).
  \end{equation}
  By \Cref{thm::1.4,eq::3}, the result follows.
\end{proof}

\section{Proof of \texorpdfstring{\Cref{thm::1.6}}{Theorem 1.2}}\label{sec::3}
In the section, we will extend some results in \cite{BPS09,Ba13} to integral oriented circulant graphs. We first introduce some related concepts and theorems.

Let $H$ be a Hermitian matrix, let $\mu_0,\ldots,\mu_{n-1}$ be all eigenvalues (not
necessarily distinct) of $H$, and let $\mathbf{u}_0,\ldots,\mathbf{u}_{n-1}$ be the corresponding normalized eigenvectors of $\mu_0,\ldots,\mu_{n-1}$ with form an orthonormal basis of $\mathbb{C}_n$. By \emph{spectral decomposition} (see~\cite[Theorem 5.5.1]{Go93}), we have $H=\sum_{r=0}^{n-1}\mu_{r} \mathbf{u}_{r}\mathbf{u}_{r}^*$, where $\mathbf{u}_{r}^*$ denotes conjugate transpose of $\mathbf{u}_{r}$. Furthermore, the transition matrix $U(t)$ of $H$ can be written as
\begin{equation*}
  U(t) =
  \sum_{r=0}^{n-1} \exp(\ii t \mu_{r}) \mathbf{u}_{r}\mathbf{u}_{r}^*.
\end{equation*}

Now, let $\Gamma=\IOCG_n(\D,\sigma)$ be an integral oriented circulant graph, and let $H$ be the Hermitian adjacency matrix of $\Gamma$. By \Cref{eq::1}, we see that  $\mathbf{u}_k=\mathbf{v}_k / \sqrt{n}$. Then the transition matrix $U(t)$ of $H$ becomes to
\begin{equation}\label{eq::4}
  U(t)=\frac1n \sum_{r=0}^{n-1} \exp(\ii \mu_r t) \mathbf{v}_r \mathbf{v}_r^*.
\end{equation}
In particular, by \Cref{eq::1} and \cref{eq::4}, we have
\begin{equation}\label{eq:7}
  U(t)_{ab}=\frac1n \sum_{r=0}^{n-1} \exp(\ii \mu_r t)
  \omega_n^{r(a-b)}=\frac 1 n \sum_{r=0}^{n-1} \exp\left(\ii \mu_r t+\ii\frac{2\pi r (a-b)}{n}\right).
\end{equation}

This expression is given in \cite[Proposition 1]{SSS07}.
Finally, our aim is to investigate whether there exist distinct integers
$a,b \in \Z_n$ and a positive real number $t$ such that
$\lvert U(t)_{ab} \rvert =1$. Let $M_r=\ii (\mu_r t+2\pi r (a-b)/n)$, for all $1 \leq r\leq n-1$. Obviously, $\lvert U(t)_{ab} \rvert \leq 1$, and equality holds if and only if for all $1 \leq r\leq n-1$, the exponents $\exp(M_r)$ are equal in \cref{eq:7}, or equivalently, $(M_{r+1}-M_{r})/(2\pi)\in \Z$. Let $t' = t/(2\pi)$. Then we have
\[
  \frac{M_{r+1}-M_{r}}{2\pi}=(\mu_{r+1}-\mu_r)t'+\frac{a-b}{n}\in \Z,
\]
for all $r=0,\ldots,n-1$. Since $\mu_r\in \Z$ for all $r=0,\ldots,n-1$, we have $t'\in \Q$.
At this point, we can obtain a necessary and sufficient condition as follows.

\begin{thm}\label{thm::3.1}
  Let $\Gamma=\IOCG_n(\D,\sigma)$ be an integral oriented circulant graph. Then for distinct $a,b \in \Z_n$, $\Gamma$ has $\PST$ between vertices $a$ and $b$
  if and only if there are integers $p$ and $q$ such that
  $\gcd(p,q)=1$ and
  \begin{equation}\label{eq:ccc}
    \frac{p}{q}(\mu_{j+1}-\mu_j)+\frac{a-b}{n}
    \in \Z,
  \end{equation}
  for all $j=0,\ldots,n-1$.
\end{thm}

By \Cref{thm::3.1}, if there exists $\PST$ on integral oriented circulant graphs, then we can easily deduce the following corollary.
\begin{cor}\label{cor::4.1}
  Let $\Gamma=\IOCG_n(\D,\sigma)$ be an integral oriented circulant graph. Then for distinct $a,b \in \Z_n$ and $1\leq k\leq n$, $\Gamma$ has $\PST$ between vertices $a$ and $b$
  if and only if there are integers $p$ and $q$ such that
  $\gcd(p,q)=1$ and
  \begin{equation}\label{eq::6}
    \frac{p}{q}(\mu_{j+k}-\mu_j)+\frac{k(a-b)}{n}
    \in \Z,
  \end{equation}
  for all $j=0,\ldots,n-1$.
\end{cor}

In general, if an integral oriented circulant graph has $\PST$ between vertices $a$ and $b$, then the order of $a-b$ is two. Furthermore, the order-two element is unique, that is, $a-b=n/2$. By~\Cref{thm::3.1}, we give a necessary and sufficient condition for the existence of $\PST$ on between vertices $b+n/2$ and $b$ on integral oriented circulant graphs.

\begin{lem} \label{lem::3.2}
  Let $\Gamma=\IOCG_n(\D,\sigma)$ be an integral oriented circulant graph. Then for all $b \in \Z_n$, $\Gamma$ has $\PST$ between vertices $b+n/2$ and $b$ if and only if there exists a
  number $m \in \N$ such that
  \begin{equation} \label{eq:a}
    \vartheta_2(\mu_{j+1}-\mu_j)=m,
  \end{equation}
  for all
  $j=0,1,\ldots, n-1$.
\end{lem}

\begin{proof}
  Let $\mu_{j+1}-\mu_{j}=2^{\alpha_{j}} m_{j}$ where $\alpha_{j}=\vartheta_{2}\left(\mu_{j+1}-\mu_{j}\right) \geq 0$ and $m_{j}$ is an odd integer for each $j=0,1, \ldots, n-1$.

  $(\Rightarrow:)$ Suppose that $\Gamma$ has $\PST$ between vertices $b+n/2$ and $b$. According to the~\Cref{thm::3.1}, there exist relatively prime integers $p, q$ such that
  \begin{equation}\label{eq:8}
    \frac{p}{q}(\mu_{j+1}-\mu_j)+\frac{1}{2}
    \in \Z,
  \end{equation}
  for all
  $j=0,1,\ldots, n-1$. Rewrite \Cref{eq:8} in the following form:
  \[
    \frac{2^{\alpha_{j}+1} \frac{p}{q} m_{j}+1}{2} \in \Z .
  \]
  From the last expression we can conclude that $2^{\alpha_{j}+1} \frac{p}{q} m_{j}$ are odd, for each $j=0,1, \ldots, n-1$.
  Since $\gcd(p, q)=1$, $p$ is odd, we can obtain $q=2^{\alpha_{j}+1}m_{q}$ and $m_{q}\mid m_{j}$, $m_{j}$ is an odd integer for each $j=0,1, \ldots, n-1$. Therefore, $\vartheta_{2}(\mu_{j+1}-\mu_{j})=\alpha_{j}=\vartheta_{2}(q)-1$, for each $j=0,1, \ldots, n-1$. Let $m=\vartheta_{2}(q)-1\in \N$. Then we can obtain $\vartheta_2(\mu_{j+1}-\mu_j)=m$.

  $(\Leftarrow:)$ Now suppose $\vartheta_2(\mu_{j+1}-\mu_j)=m$. Put $q=2^{m+1}$ and $p=1$. Then
  \[
    \frac{p\left(\mu_{j+1}-\mu_{j}\right)}{q}+\frac{1}{2}=\frac{m_{j}+1}{2} \in \Z,
  \]
  for all $j=0,1, \ldots, n-1$. Therefore, $\Gamma$ has $\PST$.
\end{proof}

\begin{lem}\label{lem::3.3}
  Let $n$ be a positive integer and let $\D$ be an odd positive integer set. Then for all odd integer $1\leq q\leq n$ and each positive integers $k$, $\sum_{d\in \D}(-1)^{k}c_{d}(q)$ have the same parity if and only if $\D=\{1\}$.
\end{lem}

\begin{proof}
  Since addition and subtraction do not affect parity. Without loss of generality, we consider $\sum_{d\in \D}c_{d}(q)$. $(\Leftarrow:)$ For set $\D=\{1\}$, since $c_{1}(q)=1$, we can derive $\sum_{d\in \D}c_{d}(q)$ have the same parity for all $q\in \N$. $(\Rightarrow:)$ For set $\D\neq\{1\}$. Suppose that there exists an odd prime set $\D'=\{d_1, d_2\}\neq \{1\}$ such that $\sum_{d\in \D'}c_{d}(q)$ have same parity for odd integer number $1\leq q\leq n$. By~\Cref{prop::2.2}. If $q=d_1$, then $\sum_{d\in \D'}c_{d}(q)=c_{d_1}(d_1)+c_{d_2}(d_1)=d_1-1-1=d_1-2\in 2\N+1$. If $q= d'$, $d_1\nmid d'$ and $d_2\nmid d'$, then $\sum_{d\in \D'}c_{d}(q)=c_{d_1}(d')+c_{d_2}(d')=-1-1=-2\in 2\N$, which leads to a contradiction.
\end{proof}
Let $\Gamma=\IOCG_n(\D,\sigma)$ be an integral oriented circulant graph. By~\Cref{thm::1.5}, the eigenvalues of $\Gamma$ are $\mu_j=  2 \sum_{d\in \D_2}\sigma(d) (-1)^{(\frac{n}{4d}-1)/2} (-1)^{(j+1)/2} c_{\frac{n}{4d}} (j)$ for $j\in 2\N+1$. By \Cref{lem::3.3}, we obtain \Cref{lem::3.4}.
\begin{lem}\label{lem::3.4}
  Let $n\equiv 0 \pmod 4$. Then $\mu_j/2$ have the same parity for $j\in 2\N+1$ if and only if $\D_2=\{n/4\}$.
\end{lem}

For $\D_2=\{n/4\}$, by \Cref{thm::1.5}, we extract the main features of the eigenvalues to obtain \Cref{lem::3.5}.
\begin{lem}\label{lem::3.5}
  Let $\Gamma=\IOCG_n(\D,\sigma)$ be an integral oriented circulant graph. If $\D_2=\{n/4\}$, then
  \begin{equation*}
    \mu_j\in
    \begin{cases}
      4\Z,                          & \text{if $j/2^{i-2}$ is odd and $i\geq 3$,} \\
      2 \sigma(n/4) (-1)^{(j+1)/2}, & \text{if $j$ is odd,}                       \\
      0,                            & \text{otherwise,}
    \end{cases}
  \end{equation*}
  for $2\leq i \leq \vartheta_2(n)$, $0\leq j \leq n-1$, where $\D=\bigcup_{i=2}^{\vartheta_2(n)}\D_i \subseteq \left\{ d: d\mid n/4\right\}$.
\end{lem}

\begin{proof}[\bf Proof of \Cref{thm::1.6}]
\ref{thm1.2.1} $\Rightarrow$ \ref{thm1.2.2}
  Suppose that $\D_2\neq\{n/4\}$. By~\Cref{lem::3.4}, $\mu_{j}/2$ have no the same parity for all $j\in 2\N+1$. This implies that $(\mu_{j+1}-\mu_{j})/2$ have no the same parity for all $j\in 2\N+1$. At this point, \Cref{eq:a} does not hold. By~\Cref{lem::3.2}, $\Gamma$ has no $\PST$, a contradiction.

  \ref{thm1.2.1} $\Leftarrow$ \ref{thm1.2.2}
  If $\D_2= \{n/4\}$, by \Cref{lem::3.5}, then $\lvert\mu_{j+1}-\mu_{j}\rvert\in 4\Z \pm 2=2(2\Z \pm 1)$ for all $0 \leq j\leq n-1$. This implies that $\vartheta_2(\mu_{j+1}-\mu_{j})=1$ for all $0 \leq j\leq n-1$. By~\Cref{lem::3.2}, $\Gamma$ has $\PST$ between vertices $b+n/2$ and $b$.
\end{proof}
From the above characterization, we can calculate the number of integral oriented circulant graphs of a given order having $\PST$.
\begin{cor}
  Let $\Gamma=\IOCG_n(\D,\sigma)$ be an integral oriented circulant graph. Then the number of $\Gamma$ having $\PST$ is
  \begin{equation*}
    \lvert\Gamma\rvert= 2\times 3^{\tau\left( \frac n 4\right) -\tau\left( \frac{n}{2^{\vartheta_2(n)}}\right)},\quad n\in 4\N,
  \end{equation*}
  where $\tau(n)$ denotes the number of the divisors of $n$.
\end{cor}

\begin{proof}
  Based on mapping $\sigma$. For $n\in
    4\N$, $d=n/4$ have two choices, the cardinality of the set $\widetilde{\D}=\{d : d \mid n/4,\ n/d\in
    4\N\}\setminus\D_2$ is equal to $\tau\left(n/4\right) -\tau\left( n/2^{\vartheta_2(n)}\right)$, and each $d$ in $\widetilde{\D}$ have two choices. According to the Binomial Theorem (see~\cite[Theorem 5.2.2]{Br10}), we can obtain the result.
\end{proof}

\section{Proof of \texorpdfstring{\Cref{thm::1.7}}{Theorem 1.3}}\label{sec::4}
In the previous section, we prove that there exists $\PST$ on integral oriented circulant graphs between vertices $b+n/2$ and $b$, and a sufficient necessary condition is also obtained. In this section, we will find $\MST$ on integral oriented circulant graphs.

\begin{lem}
  Let $\Gamma=\IOCG_n(\D,\sigma)$ be an integral oriented circulant graph. For any two distinct vertices $a,b\in \mathbb{Z}_n$, if $\Gamma$ has $\PST$ between $a$ and $b$, then $a=b+kn/4$ for some $k\in\{1,2,3\}$.
\end{lem}

\begin{proof}
  Suppose that $\Gamma$ has $\PST$ between vertices $a$ and $b$. By~\Cref{cor::4.1}, there are integers $p$ and $q$ such that
  $\gcd(p,q)=1$ and
  \begin{equation*}
    \frac{p}{q}(\mu_{j+4}-\mu_j)+\frac{4(a-b)}{n}
    \in \Z,
  \end{equation*}
  for all $j=0,\ldots,n-1$.
  By~\Cref{thm::1.6} and~\Cref{lem::3.5}, we have $\mu_{j+4}-\mu_{j}=0$ whenever $j$ is odd, and hence $4(a-b)/n\in \Z$. Therefore, $a-b\in\{n/4,n/2,3n/4\}$, and the result follows.
\end{proof}

For all $b \in \Z_n$, if $\Gamma$ has $\PST$ between vertices $b+n/4$ and $b$, by~\Cref{cor::4.1}, then \Cref{eq::6} there are integers $p$ and $q$ such that
$\gcd(p,q)=1$
\begin{equation}\label{eq:001}
  \frac{p}{q}(\mu_{j+2}-\mu_j)+\frac{2(b+n/4-b)}{n}=
  \frac{p}{q}(\mu_{j+2}-\mu_j)+\frac{1}{2}
  \in \Z,
\end{equation}
for all $j=0,\ldots,n-1$.
We can easily obtain
\begin{equation*}
  \frac{p}{q}(\mu_{j+2}-\mu_j)+\frac{3}{2}=\frac{p}{q}(\mu_{j+2}-\mu_j)+\frac{2(b+3n/4-b)}{n}\in \Z,
\end{equation*}
for all $j=0,\ldots,n-1$. Therefore, $\Gamma$ has $\PST$ between vertices $b+3n/4$ and $b$. We can find that $\Gamma$ has $\PST$ between vertices $b+n/4$ and $b$ and $\Gamma$ has $\PST$ between vertices $b+3n/4$ and $b$ are equivalent. At this point, we can obtain the following lemma.

\begin{lem}\label{thm::4.2}
  Let $\Gamma=\IOCG_n(\D,\sigma)$ be an integral oriented circulant graph. Then for all $b \in \Z_n$, $\Gamma$ has $\PST$ between vertices $b+n/4$ and $b$ and between vertices $b+n/2$ and $b$ if and only if $\Gamma$ has $\MST$ between vertices $b$, $b+n/4$, $b+n/2$, $b+3n/4$.
\end{lem}

Suppose that $\Gamma$ has $\PST$ between vertices $b+n/2$ and $b$. By~\Cref{thm::4.2}, if we want to determine the existence of $\MST$ in $\Gamma$, then we only need to satisfy that $\Gamma$ has $\PST$ between vertices $b+n/4$ and $b$.
Similar to the proof of \Cref{lem::3.2}, we can obtain \Cref{thm::4.4}.

\begin{lem}\label{thm::4.4}
  Let $\Gamma=\IOCG_n(\D,\sigma)$ be an integral oriented circulant graph. Suppose that for all $b \in \Z_n$, $\Gamma$ has $\PST$ between vertices $b+n/2$ and $b$. Then for all $b \in \Z_n$, $\Gamma$ has $\PST$ between vertices $b+n/4$ and $b$ if and only if there exists a
  number $m' \in \N$ such that
  \begin{equation}\label{eq:b}
    \vartheta_2(\mu_{j+2}-\mu_j)=m',
  \end{equation}
  for all $j=0,1,\ldots, n-1$.
\end{lem}

\begin{lem}\label{thm::4.3}
  Let $\Gamma=\IOCG_n(\D,\sigma)$ be an integral oriented circulant graph. Then for all $b \in \Z_n$, $\Gamma$ has $\MST$ between vertices $b$, $b+n/4$, $b+n/2$, $b+3n/4$ if and only if
  \begin{equation*}
    \vartheta_2(\mu_{j+1}-\mu_j)=1 \text{~and~} \vartheta_2(\mu_{j+2}-\mu_j)=2,
  \end{equation*}
  for all $j=0,1,\ldots, n-1$.
\end{lem}
\begin{proof}
  $(\Rightarrow:)$
Since $\Gamma$ has $\PST$ between vertices $b+n/2$ and $b$, according to the proof of~\Cref{thm::1.6}, we have $\vartheta_2(\mu_{j+1}-\mu_j)=1$ for all $j=0,1,\ldots, n-1$. It remains to prove  $\vartheta_2(\mu_{j+2}-\mu_j)=2$ for all $j=0,1,\ldots, n-1$. By~\Cref{thm::1.6}, we obtain $n\in 4\N$ and $\D_2=\{n/4\}$. Furthermore, by~\Cref{lem::3.5}, we have $\lvert\mu_{j+2}-\mu_{j}\rvert=4$ whenever $j$ is odd, and hence $\vartheta_2(\mu_{j+2}-\mu_j)=2$ whenever $j$ is odd. Since $\Gamma$ also has $\PST$ between vertices $b+n/4$ and $b$, by~\Cref{thm::4.4}, we conclude that $\vartheta_2(\mu_{j+2}-\mu_j)=2$ for all $j=0,1,\ldots, n-1$, as desired.

  $(\Leftarrow:)$ If $\vartheta_2(\mu_{j+1}-\mu_j)=1$ and $\vartheta_2(\mu_{j+2}-\mu_j)=2$ for all $j=0,1,\ldots, n-1$, by~\Cref{lem::3.2} and~\Cref{thm::4.4}, then $\Gamma$ has $\PST$ between vertices $b+n/2$ and $b$, and between vertices $b+n/4$ and $b$. Then it follows from~\Cref{thm::4.2}  that $\Gamma$ has $\MST$ between vertices $b$, $b+n/4$, $b+n/2$, $b+3n/4$.
\end{proof}

Let $\Gamma=\IOCG_n(\D,\sigma)$ be an integral oriented circulant graph. By~\Cref{thm::1.5}, the eigenvalues of $\Gamma$ are $\mu_j=  4 \sum_{d\in \D_3}\sigma(d) (-1)^{(\frac{n}{8d}-1)/2} (-1)^{(\frac{j}{2}+1)/2} c_{\frac{n}{8d}} (j/2)$ for $j/2 \in 2\N+1$. By~\Cref{lem::3.3}, we can obtain \Cref{lem::3.4}.

\begin{lem}\label{lem::4.4}
  Let $n\equiv 0 \pmod 4$. Then $\mu_j/4$ have the same parity for $j/2 \in 2\N+1$ if and only if $\D_3=\{n/8\}$.
\end{lem}

For $\D_2=\{n/4\}$ and $\D_3=\{n/8\}$, by \Cref{thm::1.5}, we extract the main features of the eigenvalues to obtain \Cref{lem::4.5}.
\begin{lem}\label{lem::4.5}
  Let $\Gamma=\IOCG_n(\D,\sigma)$ be an integral oriented circulant graph. If $\D_2=\{n/4\}$ and $\D_3=\{n/8\}$, then we have
  \begin{equation*}
    \mu_j\in
    \begin{cases}
      8\Z,                          & \text{if $j/2^{i-2}$ is odd and $i\geq 4$,} \\
      4\sigma(n/8)(-1)^{(j/2+1)/2}, & \text{if $j/2$ is odd,}                     \\
      2\sigma(n/4)(-1)^{(j+1)/2},   & \text{if $j$ is odd,}                       \\
      0,                            & \text{otherwise,}
    \end{cases}
  \end{equation*}
  for $2\leq i \leq \vartheta_2(n)$, $0\leq j \leq n-1$, where $\D=\bigcup_{i=2}^{\vartheta_2(n)}\D_i \subseteq \left\{ d: d\mid n/4\right\}$.
\end{lem}

\begin{proof}[\bf Proof of \Cref{thm::1.7}]
  By~\Cref{thm::4.2} and \Cref{thm::1.6}, $\D_2=\{n/4\}$ is a necessary condition for the existence of $\MST$ on integral oriented circulant graphs. Next, let $\D_2=\{n/4\}$. We prove that $\D_3= \{n/8\}$.
  \ref{thm1.2.1} $\Rightarrow$ \ref{thm1.2.2}
  Suppose that $\D_3\neq\{n/8\}$. By~\Cref{lem::4.4}, then $\mu_{j}/4$ have no the same parity for all $j/2 \in 2\N+1$. This implies that $(\mu_{j+1}-\mu_{j})/4$ have no the same parity for all $j/2 \in 2\N+1$. At this point, \Cref{eq:b} does not hold. By~\Cref{thm::4.4} and~\ref{thm::4.3}, $\Gamma$ has no $\MST$, a contradiction.

  \ref{thm1.2.1} $\Leftarrow$ \ref{thm1.2.2}
  If $\D_3= \{n/8\}$, then we have $n\in 8\N$. By \Cref{lem::4.5},
  then $\lvert\mu_{j+2}-\mu_{j}\rvert=4$ for all $j$ is odd, and
  \[
    \lvert\mu_{j+2}-\mu_{j}\rvert\in
    \begin{cases}
      4(2\Z\pm 1), & \text{if $j/2^{i-2}$ is odd,} \\
      4,           & \text{otherwise,}
    \end{cases}
  \]
  for all $j$ is even. This implies that $\vartheta_2(\mu_{j+2}-\mu_{j})=2$ for all $0 \leq j\leq n-1$. By~\Cref{thm::4.3}, $\Gamma$ has $\MST$ between vertices $b$, $b+n/4$, $b+n/2$, $b+3n/4$.
\end{proof}

By~\Cref{thm::1.7}, we can obtain that $\MST$ on integral oriented circulant graphs only occurs between four vertices and $n\in 8\N$. Therefore, we have the following corollary.

\begin{cor}
  Let $\Gamma=\IOCG_n(\D,\sigma)$ be an integral oriented circulant graph. Then $\Gamma$ has no $\UST$.
\end{cor}

From the above characterization, we can calculate the number of integral oriented circulant graphs of a given order having $\MST$.
\begin{cor}
  Let $\Gamma=\IOCG_n(\D,\sigma)$ be an integral oriented circulant graph. Then the number of $\Gamma$ having $\MST$ is
  \begin{equation*}
    \lvert\Gamma\rvert= 2\times 2 \times 3^{\tau\left( \frac n 4\right) -2\tau\left( \frac{n}{2^{\vartheta_2(n)}}\right)},\quad n\in 8\N,
  \end{equation*}
  where $\tau(n)$ denotes the number of the divisors of $n$.
\end{cor}

\begin{proof}
  Based on mapping $\sigma$. For $n\in
    8\N$,  $d=n/4$ and $d=n/8$ have two choices respectively, the cardinality of the set $\widetilde{\D}=\{d : d \mid n/4,\ n/d\in
    8\N\}\setminus\{\D_2,\D_3\}$ is equal to $\tau\left(n/4\right) -2\tau\left( n/2^{\vartheta_2(n)}\right)$, and each $d$ in $\widetilde{\D}$ have two choices. According to the Binomial Theorem (see~\cite[Theorem 5.2.2]{Br10}), we can obtain the result.
\end{proof}

\section{Conclusion}
This work focuses on the study of existence problem for $\PST$ and $\MST$ on integral oriented circulant graphs. We find that there are some nice $\PST$ and $\MST$ properties on integral oriented circulant graphs. $\PST$ and $\MST$ determined by its order $n$ and the set of divisors $\D$, not related to the selection of mapping $\sigma$. For $\PST$ (or $\MST$), we obtain necessary and sufficient condition for the existence of $\PST$ (or $\MST$) on integral oriented circulant graphs. For $\MST$, we prove that there only exists $\MST$ for four vertices on integral oriented circulant graphs.

The following question now arises naturally:
\begin{enumerate}[label = \bf(\roman*)]
  \item Determine non-integer mixed (or oriented)  Cayley (or circulant) graphs having $\PST$ and $\MST$.
  \item Determine integral (weighted) mixed circulant graphs having $\PST$ and $\MST$. According to the calculation, we find that there are many $\PST$ on integral mixed circulant graphs.
  \item Determine the mixed Cayley graphs with $\PST$ for some kinds of groups.
  \item So far, we only find that $\MST$ can occur between three vertices and four vertices. Is there a $\MST$ with more than four vertices?
\end{enumerate}


\bibliographystyle{elsarticle-num-names-nourl}

\bibliography{IOCG.bib}
\end{document}